\documentclass{amsart}
\usepackage{amsfonts,amssymb,amscd,amsmath,enumerate,verbatim,calc,latexsym,pstcol,pst-plot,pst-3d}

\def\NZQ{\mathbb}               
\def\NN{{\NZQ N}}
\def\QQ{{\NZQ Q}}
\def\ZZ{{\NZQ Z}}

\def\FF{{\NZQ F}}

\def\frk{\mathfrak}               

\def\mm{{\frk m}}

\def\Phi{{\frk N}}
\def\ab{{\bold a}}
\def\bb{{\bold b}}
\def\xb{{\bold x}}

\def\opn#1#2{\def#1{\operatorname{#2}}} 
\opn\chara{char} \opn\length{\ell} \opn\pd{pd} \opn\rk{rk}
\opn\projdim{proj\,dim} \opn\injdim{inj\,dim} \opn\rank{rank}
\opn\depth{depth} \opn\grade{grade} \opn\height{height}
\opn\embdim{emb\,dim} \opn\codim{codim}

\opn\Tr{Tr} \opn\bigrank{big\,rank}
\opn\superheight{superheight}\opn\lcm{lcm}
\opn\trdeg{tr\,deg}
\opn\reg{reg} \opn\lreg{lreg} \opn\ini{in} \opn\lpd{lpd}
\opn\size{size}\opn{\mult}{mult}
\opn\link{lk}\opn\star{st}
\opn\div{div} \opn\Div{Div} \opn\cl{cl} \opn\Cl{Cl}
\opn\Spec{Spec} \opn\Supp{Supp} \opn\supp{supp} \opn\Sing{Sing}
\opn\Ass{Ass} \opn\Min{Min}
\opn\Ann{Ann} \opn\Rad{Rad} \opn\Soc{Soc}
\opn\Syz{Syz} \opn\Im{Im} \opn\Ker{Ker} \opn\Coker{Coker}
\opn\Am{Am} \opn\Hom{Hom} \opn\Tor{Tor} \opn\Ext{Ext}
\opn\End{End} \opn\Aut{Aut} \opn\id{id}

\opn\nat{nat}
\opn\pff{pf}
\opn\Pf{Pf} \opn\GL{GL} \opn\SL{SL} \opn\mod{mod} \opn\ord{ord}
\opn\Gin{Gin}
\opn\Hilb{Hilb}\opn\adeg{adeg}\opn\std{std}\opn\ip{infpt}
\opn\Pol{Pol}
\opn\sat{sat}
\opn\Var{Var}

\opn\aff{aff} \opn\con{conv} \opn\relint{relint} \opn\st{st}
\opn\lk{lk} \opn\cn{cn} \opn\core{core} \opn\vol{vol}
\opn\link{link} \opn\star{star} \opn\car{char}
\opn\gr{gr}

\def\Fc{{\mathcal F}}
\def\pot#1#2{#1[\kern-0.28ex[#2]\kern-0.28ex]}

\opn\dirlim{\underrightarrow{\lim}}
\opn\inivlim{\underleftarrow{\lim}}
\let\sect=\cap

\let\Dirsum=\bigoplus

\let\to=\rightarrow

\def\Implies{\ifmmode\Longrightarrow \else
        \unskip${}\Longrightarrow{}$\ignorespaces\fi}
\def\implies{\ifmmode\Rightarrow \else
        \unskip${}\Rightarrow{}$\ignorespaces\fi}
\def\iff{\ifmmode\Longleftrightarrow \else
        \unskip${}\Longleftrightarrow{}$\ignorespaces\fi}

\let\:=\colon
\newtheorem{Theorem}{Theorem}[section]
\newtheorem{Theorem A}{Theorem}
\newtheorem{Lemma}[Theorem]{Lemma}
\newtheorem{Corollary}[Theorem]{Corollary}
\newtheorem{Proposition}[Theorem]{Proposition}

\newtheorem{Example}[Theorem]{Example}

\newtheorem{Definition}[Theorem]{Definition}

\let\epsilon\varepsilon
\let\phi=\varphi
\let\kappa=\varkappa
\def\qed{\ifhmode\textqed\fi
      \ifmmode\ifinner\quad\qedsymbol\else\dispqed\fi\fi}
\def\textqed{\unskip\nobreak\penalty50
       \hskip2em\hbox{}\nobreak\hfil\qedsymbol
       \parfillskip=0pt \finalhyphendemerits=0}
\def\dispqed{\rlap{\qquad\qedsymbol}}

\opn\dis{dis}
\def\pnt{{\raise0.5mm\hbox{\large\bf.}}}

\opn\Lex{Lex}

\begin{document}

\title{Simplicial complexes with rigid depth}

\author{Adnan Aslam and Viviana Ene}

\address{ Abdus Salam School of Mathematical Sciences (ASSMS), GC University, Lahore, ~Pakistan.
}
\email{adnanaslam15@yahoo.com}

\address{Faculty of Mathematics and Computer Science, Ovidius University, Bd.\ Mamaia 124,
 900527 Constanta, Romania
  \newline \noindent Institute of Mathematics of the Romanian Academy, P.O. Box 1-764, RO-014700, Buchaest, Romania} \email{vivian@univ-ovidius.ro}

\begin{abstract} We extend a result of Minh and Trung \cite{MT} to get criteria for  $\depth I=\depth\sqrt{I}$ where $I$ is an unmixed monomial ideal of the polynomial ring $S=K[x_1,\ldots, x_n]$. As an application we characterize all the pure simplicial complexes $\Delta$ which have rigid depth, that is, which satisfy the condition that  for  every unmixed monomial ideal
$I\subset S$ with $\sqrt{I}=I_\Delta$ one has $\depth(I)=\depth(I_\Delta).$ 
\end{abstract}

\thanks{The second author was supported by the grant UEFISCDI,  PN-II-ID-PCE- 2011-3-1023.}

\maketitle

\section*{Introduction}

Let $S=K[x_1,\ldots,x_n]$ be the polynomial ring over a field $K$ and $I\subset S$ a monomial ideal. In \cite{HTT}, the authors compare the properties 
of $I$ with the properties of its radical by using the inequality $\beta_i(I)\geq \beta_i(\sqrt{I}).$ In particular, from the inequality between the 
Betti numbers, one gets the inequality $\depth(S/I)\leq \depth(S/\sqrt{I}),$ which implies, for instance, that $S/I$ is Cohen-Macaulay if $S/\sqrt{I}$
is so. In \cite{MT}, the authors presented criteria for the Cohen-Macaulayness of a monomial ideal in terms of its primary decomposition. We extend 
their criteria to characterize the unmixed monomial ideals for which the equality $\depth(S/I)= \depth(S/\sqrt{I})$ holds. We recall that an ideal $I\subset S$ is \textit{unmixed} if the associated prime ideals of $S/I$ are the minimal prime ideals of $I.$

Let $\Delta$ be a pure simplicial complex  with the facet set denoted, as usual, by  $\Fc(\Delta)$, and let $I_{\Delta}=\bigcap_{F\in \Fc(\Delta)}P_{F}$ 
be  its Stanley-Reisner ideal. For any subset $F\subset [n],$ we denoted by $P_F$ the monomial prime ideal generated by the variables $x_i$ 
with $i\notin F$. Let $I\subset S$ be an unmixed monomial ideal such that $\sqrt{I}=I_{\Delta}$ and assume that $I=\bigcap_{F\in \Fc(\Delta)} I_{F}$ where 
$I_{F}$ is the $P_{F}$-primary component of $I.$ Following \cite{MT}, for every $\ab\in \NN^n$, $\ab=(a_1,\ldots,a_n),$ we set $\xb^\ab=x_1^{a_1}\cdots x_n^{a_n}$ and denote by $\Delta_\ab$ the simplicial complex on the set $[n]$ with the facet set $\Fc(\Delta_\ab)=\{F\in \Fc(\Delta)\ |\  \xb^\ab\notin I_F\}$. Moreover, for every simplicial complex $\Gamma$ with $\Fc(\Gamma)\subseteq \Fc(\Delta),$ we set 
\[
L_\Gamma(I)=\{\ab\in \NN^n\ |\  \xb^\ab \in \bigcap_{F\in \Fc(\Delta)\setminus \Fc(\Gamma)}I_F\setminus \bigcup_{G\in \Fc(\Gamma)}I_G\}. 
\]

In Section~\ref{extensionsect}, we prove the following theorem which is a natural extension of Theorem~1.6 in  \cite{MT}.

\begin{Theorem A}
Let $\Delta$ be a pure simplicial complex  with $\depth K[\Delta]=t.$ Let $I\subset S$ be an unmixed monomial ideal with $\sqrt{I}=I_\Delta$. Then the following conditions are equivalent:
\begin{itemize}
	\item [(a)] $\depth(S/I)=\depth(S/\sqrt{I}),$
	\item [(b)] $\depth K[\Delta_\ab]\geq t$ for all $\ab\in \NN^n,$
	\item [(c)] $L_{\Gamma}(I)=\emptyset$ for every simplicial complex $\Gamma$ with $\Fc(\Gamma)\subseteq \Fc(\Delta)$ and $\depth K[\Gamma]<t.$
\end{itemize}
\end{Theorem A}

As a main application of the above theorem we study in Section~\ref{rigid} a  special class of simplicial complexes. We say that  a pure simplicial complex has \textit{rigid depth} if for every unmixed monomial ideal $I\subset S$ with $\sqrt{I}=I_{\Delta}$ one has $\depth(S/I)=\depth(S/I_\Delta).$ 
In Theorem~\ref{rigiddepth} which generalizes \cite[Theorem 3.2]{HTT}, we give necessary and sufficient conditions for $\Delta$ to have rigid depth. In particular, from this characterization, it follows that if a pure simplicial complex has rigid depth over a field of characteristic $0,$ then it has rigid depth over any field. In the last part  we  discuss the behavior of rigid depth  in connection to  the skeletons of the simplicial complex.

\section{Criteria for $\depth(S/I)=\depth(S/\sqrt{I})$} 
\label{extensionsect}

Let $S=K[x_1,\ldots,x_n]$ be the polynomial ring over a field $K$.  Let $I\subset S$ be an unmixed monomial ideal such that $\sqrt{I}=I_{\Delta}$ where $\Delta$ is a pure simplicial complex  with the facet set $\Fc(\Delta).$ Then $I_{\Delta}=\bigcap_{F\in \Fc(\Delta)}P_F$, where $P_F=(x_i\ |\  i\notin F)$ for every $F\in \Fc(\Delta).$ Let $I=\bigcap_{F\in \Fc(\Delta)}I_F$ where $I_F$ is the $P_F$-primary component of $I.$

In order to prove the main result of this section we need to recall some facts from \cite[Section 1]{MT}. For $\ab=(a_1,\ldots,a_n)\in \ZZ^n,$ let
$G_\ab=\{i\ |\  a_i<0\}$. We denote by $\Delta_\ab$ the simplicial complex on $[n]$ of all the sets of the form $F\setminus G_\ab$ where
$G_\ab\subset F\subset [n]$ and such that $F$ satisfies the condition $\xb^\ab\notin IS_F$ where $S_F=S[x_i^{-1}\ |\  i\in F]$.  It is shown in \cite[Section 1]{MT}  that if $\Delta_\ab$ is non-empty, then $\Delta_\ab$ is a pure 
subcomplex of $\Delta$ of $\dim \Delta_\ab=\dim \Delta-|G_\ab|.$ 

For every simplicial subcomplex $\Gamma$ of $\Delta$ with $\Fc(\Gamma)\subset \Fc(\Delta)$ we set 
\[
L_\Gamma(I)=\{\ab\in \NN^n\ |\  \xb^\ab \in \bigcap_{F\in \Fc(\Delta)\setminus \Fc(\Gamma)}I_F\setminus \bigcup_{G\in \Fc(\Gamma)}I_G\}. 
\]

By \cite[Lemma 1.5]{MT}, we have
\begin{equation}\label{ii}\Delta_\ab=\Gamma \textnormal{ if and only if } \ab\in L_\Gamma(I). \end{equation}

For the proof of the next theorem we also need to recall Takayama's formula \cite{T}. For every degree $\ab\in \ZZ^n$ we denote by 
$H_\mm^i(S/I)_\ab$ the $\ab$-component of the $i$th local cohomology module of $S/I$ with respect to the homogeneous maximal ideal of $S.$ For $1\leq j\leq n,$ let $$\rho_j(I)=\max\{\nu_j(u)\ |\  u \textnormal { is a minimal generator of } I\},$$ where  by $\nu_j(u)$ we mean the exponent of the variable $x_j$ in $u$. If $x_j$ does not divide $u,$ then we use the usual convention, $\nu_j(u)=0.$

\begin{Theorem}[Takayama's formula]\label{Takayama}
\[\dim_K H_\mm^i(S/I)_\ab=\left\{
\begin{array}{ll}
	\dim_K \tilde{H}_{i-|G_\ab|-1}(\Delta_\ab, K), & \textnormal{ if } G_\ab\in \Delta \textnormal{ and }\\
	&  a_j< \rho_j(I) \textnormal{ for } 1\leq j\leq n,\\
	0, & \textnormal { else. }
\end{array}\right.
 \]
\end{Theorem}

The next theorem is a  natural extension of \cite[Theorem 1.6]{MT}.

\begin{Theorem}\label{extension}
Let $\Delta$ be a pure simplicial complex  with $\depth K[\Delta]=t.$ Let $I\subset S$ be an unmixed monomial ideal with $\sqrt{I}=I_\Delta$. The  following conditions are equivalent:
\begin{itemize}
	\item [(a)] $\depth (S/I)=t,$
	\item [(b)] $\depth K[\Delta_\ab]\geq t$ for all $\ab\in \NN^n$ with $\Delta_\ab\neq \emptyset,$
	\item [(c)] $L_{\Gamma}(I)=\emptyset$ for every simplicial complex $\Gamma$ with $\Fc(\Gamma)\subseteq \Fc(\Delta)$ and $\depth K[\Gamma]<t.$
\end{itemize}
\end{Theorem}

\begin{proof} The proof of this theorem follows closely the ideas of the proof of \cite[Theorem~1.6]{MT}.
For the equivalence $(a)\Leftrightarrow (b)$ we need to recall some known facts about  local cohomology; see 
\cite[Section A. 7]{HH}. For any finitely generated graded $S$-module $M$ we have $\depth M\geq t$ if and only if $H_\mm^i(M)=0$ for all $i<t.$ Therefore, in our hypothesis, and since $\depth(S/I)\leq \depth(S/\sqrt{I})=t,$ we get
\begin{equation}\label{eq1}
\depth(S/I)=t \Leftrightarrow H_\mm^i(S/I)=0 \text{ for } i<t.
\end{equation}

In addition, for every $\ab\in \NN^n$, we get 
\begin{equation}\label{eq2}
\depth(K[\Delta_\ab])\geq t \Leftrightarrow H_\mm^i(K[\Delta_\ab])=0 \text{ for } i<t.
\end{equation}

For $\bb\in \ZZ^n,$ we set $G_\bb=\{i\ | \ b_i<0\}$ and $H_\bb=\{i\ |\ b_i>0\}$. By using \cite[Theorem A.7.3]{HH}, for every $\bb\in \ZZ^n,$ we obtain
\[
\dim_K H_\mm^i(K[\Delta_\ab])_\bb=\dim_K\tilde{H}_{i-|G_\bb|-1}(\link_{\star H_\bb}G_\bb;K).
\]
 Here we denoted by $\star H_\bb$ the star of $H_\bb$ in 
$\Delta_\ab$,  and by $\link_{\star H_\bb}G_\bb$ the link of $G_\bb$ in the complex $\star H_\bb$. We recall that if $\Gamma$ is a simplicial complex and $F$ is a face of $\Gamma,$ then $\star_{\Gamma}F=\{G\ |\ F\cup G\in \Gamma\}$ and $\link_\Gamma F=\{G\ |\ F\cup G\in \Gamma \text{ and }F\cap G=\emptyset\}.$ Therefore, the equivalence (\ref{eq2}) my be written
\begin{equation}\nonumber
\depth K[\Delta_\ab]\geq t
\end{equation}
\begin{equation}\label{eq3}
 \Leftrightarrow \tilde{H}_{i-|G_\bb|-1}(\link_{\star H_\bb}G_\bb;K)=0 \text{ for } i<t \text{ and for every }\bb\in\ZZ^n.
\end{equation}

Since $\link_{\star H_\bb}G_\bb$ is acyclic for $H_\bb \neq \emptyset$ and  $\star H_\bb=\Delta_\ab$ if $H_\bb=\emptyset$,  we get
\begin{equation}\nonumber
\depth K[\Delta_\ab]\geq t
\end{equation}
\begin{equation}\label{eq4}
 \Leftrightarrow \tilde{H}_{i-|G_\bb|-1}(\link_{\Delta_\ab}G_\bb;K)=0 \text{ for } i<t \text{ and for every }\bb\in\ZZ^n.
\end{equation}

By Takayama's formula, the equivalence (\ref{eq1}) may be rewritten
\begin{equation}\nonumber
\depth(S/I)=t 
\end{equation}
\begin{equation}\label{eq5}
\Leftrightarrow \dim_K\tilde{H}_{i-|G_\bb|-1}(\Delta_\bb;K)=0 \text{ for } i<t \text{ and for every }\bb\in\ZZ^n.
\end{equation}

Now, the equivalence (a)$\Leftrightarrow$ (b) follows by  relations (\ref{eq4}) and (\ref{eq5}) if we notice that, by the proof of
(i) $\Rightarrow$ (ii) in \cite[Theorem 1.6]{MT}, we have $\link_{\Delta_\ab}G_\bb=\Delta_\bb$ for any $G_\bb \in \Delta_\ab.$\\

For the rest of the proof we only need to use (\ref{ii}). Indeed, for (b) $\Rightarrow$ (c), let us assume that $L_\Gamma(I)\neq\emptyset$ for some subcomplex $\Gamma$ of $\Delta$ with $\Fc(\Gamma)\subset \Fc(\Delta)$ and such that $\depth(K[\Gamma])<t.$ Then there exists $\ab\in\L_\Gamma(I)$,
 hence $\Gamma=\Delta_\ab.$ But this equality is impossible since $\depth(K[\Delta_\ab])\geq t.$ For (c) $\Rightarrow$ (b), let us assume that there 
 exists $\ab\in \NN^n$ such that $\depth K[\Delta_\ab]<t$. Then, for $\Gamma=\Delta_\ab$ we get $L_\Gamma(I)\neq\emptyset$, a contradiction. 
\end{proof}

Obviously, for $t=\dim K[\Delta]$ in the above theorem we recover Theorem 1.6 in \cite{MT}.\\

The above theorem is especially useful in the situation when $I$ is either an intersection of monomial prime ideal powers or an intersection of irreducible monomial ideals. The first class of ideals may be studied with completely similar arguments to those used in \cite[Section 1]{MT}. In the sequel we discuss ideals which are intersections of irreducible monomial ideals.

Let $\Fc(\Delta)=\{F_1,\ldots,F_r\}$ and $I=\bigcap_{i=1}^r I_{F_i}$ be an intersection of irreducible monomial ideals, that is, for every $1\leq i\leq r,$ 
$I_{F_i}=(x_j^{a_{ij}}\ |\ j\not\in F_i)$ for some positive exponents $a_{ij}$. As a consequence of the above theorem, one may express the condition 
$\depth(S/I)=\depth(S/\sqrt{I})$ in terms of linear inequalities on the exponents $a_{ij}.$

\begin{Proposition}\label{inequalities}
The set of exponents $(a_{ij})$ for which the equality $\depth(S/I)=\depth(S/\sqrt{I})$ holds consists of all points of positive integer coordinates in a  finite union of rational cones in ${\mathbb{R}}^{r(n-d)}$.
\end{Proposition}

\begin{proof}
Let $\Gamma$ be a subcomplex of $\Delta$ with $\depth(K[\Gamma])<t$ and $\Fc(\Delta)\setminus \Fc(\Gamma)=\{F_{i_1},\ldots,F_{i_s}\}$ where
$1\leq i_1< \cdots < i_s\leq r$. The condition $L_{\Gamma}(I)=\emptyset$ gives
\[
\bigcap_{q=1}^s(x_j^{a_{{i_q}j}}: j\notin F_{i_q})\subseteq \bigcup_{k\notin \{i_1,\ldots,i_s\}} I_{F_k}.
\]
This implies that the following conditions must hold
\[
\lcm(x_{j_1}^{a_{i_1j_1}},x_{j_2}^{a_{i_2j_2}},\ldots ,x_{j_s}^{a_{i_sj_s}})\in \bigcup_{k\notin \{i_1,\ldots,i_s\}} I_{F_k}
\]
for all $s$-tuples $(j_1,j_2,\ldots,j_s)$, with $j_q \notin F_{i_q}$ for $1\leq q \leq s$. This is equivalent to saying that
for every $s$-tuple $(j_1,j_2,\ldots,j_s)$, with $j_q \notin F_{i_q}$ for $1\leq q \leq s$, there exists $1\leq q \leq s$ such that
\[
 a_{i_qj_q}\geq \min \{a_{kj_q}: k \neq i_1,i_2,\ldots,i_s,\}.
\]
\end{proof}

In the following example we consider tetrahedral type ideals.

\begin{Example}{\em
Let $\Delta$ be the $4$-cycle, that is, $I_\Delta=(x_1,x_2)\sect(x_1,x_4)\sect (x_2,x_3)\sect (x_3,x_4).$
 Note that $S/I_\Delta$ is Cohen-Macaulay, hence $\depth(S/I_\Delta)=2.$



Let $I=(x_1^{a_1},x_2^{a_2})\sect (x_1^{a_3},x_4^{a_4})\sect (x_2^{a_5},x_3^{a_6})\sect (x_3^{a_7},x_4^{a_{8}})$.
Then  $\depth(S/I)=\depth(S/I_\Delta)$, that is, $I$ is a Cohen-Macaulay ideal, if and only if
one of the following condition holds:
\begin{enumerate}
\item[{ (1)}] $a_{3} \leq a_{1},\quad a_{2}=a_{5},\quad a_{7}\leq a_{6}$.
\item[{ (2)}] $a_{2} \leq a_{5},\quad a_{6}=a_{7},\quad a_{4}\leq a_{8}$.
\item[{ (3)}] $a_{5} \leq a_{2},\quad a_{1}=a_{3},\quad a_{8}\leq a_{4}$.
\item[{ (4)}] $a_{1} \leq a_{3},\quad a_{4}=a_{8},\quad a_{6}\leq a_{7}$.
\end{enumerate}

 In order to prove the above claim, we first notice that any subcomplex  $\Gamma$ of $\Delta$ which has $\depth(K[\Gamma])<2$ corresponds to a  disconnected subgraph of $\Delta$. But $\Delta$ has two disconnected subgraphs which correspond to the pair of disjoint edges $\big\{\{1,2\},\{3,4\}\big\}$ and $\big\{\{1,4\},\{2,3\}\big\}$. Let $\Gamma$ be the subgraph $\big\{\{1,2\},\{3,4\}\big\}$. Then the inequalities of the proof of Proposition~\ref{inequalities} give
 \[
 (a_1\leq a_3 \text{ or } a_2\leq a_5) \text{ and } (a_1\leq a_3 \text{ or }a_7\leq a_6)
 \]
 \[\text{ and }(a_8\leq a_4 \text{ or }a_2\leq a_5)
 \text{ and } (a_8\leq a_4 \text{ or } a_7\leq a_6 ),
 \]
which  is equivalent to
\begin{equation}\label{cond1}
 (a_1\leq a_3 \text{ and } a_8 \leq a_4) \quad \text{or} \quad (a_2\leq a_5 \text{ and }  a_7 \leq a_6).
\end{equation}

Now we consider the other disconnected subgraph which corresponds to the pair of disjoint edges $\big\{\{1,4\},\{2,3\}\big\}$ and get, similarly,
\begin{equation}\label{cond2}
(a_3\leq a_1 \text{ and } a_5 \leq a_2) \quad \text{or} \quad (a_6\leq a_7 \text{ and } a_4 \leq a_8).
\end{equation}
By intersecting conditions (\ref{cond1}) and (\ref{cond2}), we get the desired relations.

Note that in this example the union of the four rational cones defined by the set of the linear inequalities $(1)-(4)$ is not a convex set. Indeed, if we take the exponent vectors  $\ab=(3,5,1,3,5,9,7,9)$ and $\ab^\prime=(1,3,1,1,7,11,11,1)$, then the corresponding ideals are both Cohen-Macaulay. However, for the vector $\bb=(\ab+\ab^\prime)/2=(2,4,1,2,6,10,9,5)$,  the corresponding ideal is not Cohen-Macaulay.
}
\end{Example}

\section{Rigid depth}
\label{rigid}

\begin{Definition} {\em 
Let $\Delta$ be a pure simplicial complex. We say that $\Delta$ has \textit{rigid depth} if for every unmixed monomial  ideal
$I\subset S$ with $\sqrt{I}=I_\Delta$ one has $\depth(S/I)=\depth(S/I_\Delta).$ }
\end{Definition}

For example, any pure simplicial complex $\Delta$ with $\depth(K[\Delta])=1$ has rigid depth. In this section we characterize all the pure simplicial complexes which have rigid depth. 

In the next theorem we will use the formula given in the following proposition for computing the depth of a Stanley-Reisner ring. We recall that the $i$th skeleton of a simplicial complex $\Delta$ is defined as  $\Delta^{(i)}=\{F\in \Delta\ |\ \dim F\leq i\}.$

\begin{Proposition}\cite{H}\label{Hibi}
Let $\Delta$ be a simplicial complex of dimension $d-1$. Then:
\[
\depth(K[\Delta])=\max\{i\ |\ \Delta^{(i)} \text{ is Cohen-Macaulay}\}+1.
\]

\end{Proposition}

The following theorem generalizes \cite[Theorem 3.2]{HTT}.

\begin{Theorem}\label{rigiddepth}
Let $\Delta$ be a pure simplicial complex  with $\depth(K[\Delta])=t$ and $I_\Delta=\bigcap_{F\in \Fc(\Delta)}P_F.$ The following statements are equivalent:
\begin{itemize}
	\item [(a)] $\Delta$ has rigid depth.
	\item [(b)] $\depth(S/I)=t$ for every ideal $I=\bigcap_{F\in \Fc(\Delta)}I_F$ where $I_F$ are irreducible monomial ideals with $\sqrt{I_F}=P_F$ for all $F\in \Fc(\Delta)$.
	\item [(c)] $\depth(S/I)=t$ for every ideal $I=\bigcap_{F\in \Fc(\Delta)}P^{m_F}_F$ where $m_F$ are positive integers.
	\item [(d)] $\depth(K[\Gamma])\geq t$ for every subcomplex $\Gamma$ of $\Delta$ with $\Fc(\Gamma)\subset \Fc(\Delta)$.
	\item [(e)] For every subcomplex $\Gamma$ of $\Delta$ with $\Fc(\Gamma)\subset \Fc(\Delta)$, the skeleton $\Gamma^{(t-1)}$ is Cohen-Macaulay. 
	\item [(f)] Let $\Fc(\Delta)=\{F_1,\ldots, F_r\}$. Then, for every $1\leq k\leq \min\{r,t\}$ and for any indices  $1\leq i_1<\cdots <i_k\leq r$,  we have $|F_{i_1}\cap\cdots\cap F_{i_k}|\geq t-k+1.$
\end{itemize}
\end{Theorem}

\begin{proof}
(a) $\Rightarrow$ (b) and (a) $\Rightarrow$ (c) are trivial.

(b) $\Rightarrow$ (d): Let $\Gamma$ be a subcompex of $\Delta$ with $\Fc(\Gamma)\subset\Fc(\Delta)$. We have to show that $\depth(K[\Gamma])\geq t.$
For every $F\in \Fc(\Gamma)$, let $I_F=(x_i^2\ |\ i\notin F)$, and for every $F\in \Fc(\Delta)\setminus\Fc(\Gamma)$ let $I_F=P_F=(x_i\ |\ i\notin F)$.
Let $I=\bigcap_{F\in \Fc(\Delta)}I_F$. By assumption, $\depth(S/I)=t$. Let $S^\prime\subset K[x_1,\ldots,x_n,y_1,\ldots,y_n]$ be the polynomial ring over $K$ in all the variables which are needed for the polarization of $I,$ and let $I^p\subset S^\prime$ be the polarization of $I.$ We have 
$I^p=\bigcap_{F\in \Fc(\Delta)}I_F^p$, where 
\[
I_F^p=
\left\{
\begin{array}{ll}
(x_iy_i\ |\ i\notin F),& \text{ if } F\in \Fc(\Gamma),\\
P_F,& \text{ if } F\in\Fc(\Delta)\setminus \Fc(\Gamma).
\end{array}
\right.
\]
 Then $\projdim(S^\prime/I^p)=
\projdim(S/I)$. Let $N$ be the multiplicative set generated by all the variables $x_i.$ Then $I_N^p=\bigcap_{F\in \Fc(\Gamma)}(y_i\ |\ i\notin F)$ and $$\projdim(S^\prime/I^p)_N\leq \projdim(S^\prime/I^p)=\projdim(S/I).$$ This inequality implies that $\depth(K[\Gamma])\geq \depth(S/I)=t.$

(d) $\Leftrightarrow$ (e) follows immediately by applying the criterion given in Proposition~\ref{Hibi}.

(d) $\Rightarrow$ (f): We proceed by induction on $k.$ The initial inductive step is trivial. Let $k>1$ and assume that $|F_{i_1}\cap\cdots \cap F_{i_\ell}|\geq t-\ell+1$ for $1\leq \ell < k$ and for any $1\leq i_1<\cdots <i_\ell\leq r.$ Obviously, it is enough to show that 
$|F_1\cap\cdots\cap F_k|\geq t-k+1.$ By \cite[Theorem 1.1]{HPV}, we have the following exact sequence of $S$-modules: 
{\small
\begin{equation}\label{sequence}
0\to \frac{S}{\bigcap_{i=1}^k P_{F_i}}\to\Dirsum\limits_{i=1}^k \frac{S}{P_{F_i}}\to\Dirsum\limits_{1\leq i<j\leq k}\frac{S}{P_{F_i}+P_{F_j}}\to\cdots
\to\frac{S}{P_{F_1}+\cdots + P_{F_k}}\to 0.
\end{equation}}
By assumption,  $\depth(S/\bigcap_{i=1}^k P_{F_i})\geq t$. We decompose the above sequence  in $k-1$ short exact sequences as follows:
\[
0\to \frac{S}{\bigcap_{i=1}^k P_{F_i}}\to\Dirsum\limits_{i=1}^k \frac{S}{P_{F_i}}\to U_1\to 0,
\]
\[
0\to U_1\to \Dirsum\limits_{1\leq i<j\leq k}\frac{S}{P_{F_i}+P_{F_j}}\to U_2\to 0,
\]
\[
\vdots
\]
\[
0\to U_{k-2}\to \Dirsum\limits_{1\leq j_1<\cdots < j_{k-1}\leq k}\frac{S}{P_{F_{j_1}}+\cdots P_{F_{j_{k-1}}}}\to \frac{S}{P_{F_1}+\cdots + P_{F_k}}\to 0.
\]
Note that, for all $\ell$ and any $1\leq j_1<\cdots< j_\ell\leq k,$ we have $$P_{F_{j_1}}+\cdots + P_{F_{j_\ell}}=P_{F_{j_1}\cap\cdots\cap F_{j_\ell}}.$$ In particular, $S/(P_{F_{j_1}}+\cdots + P_{F_{j_\ell}})$ is Cohen-Macaulay of depth equal to $|F_{j_1}\cap\cdots\cap F_{j_\ell}|.$ Therefore, $$\depth(\Dirsum_{1\leq j_1<\cdots<j_\ell\leq k}S/(P_{F_{j_1}}+\cdots P_{F_{j_\ell}}))\geq t-\ell+1$$ for every $1\leq \ell<k$ and any $1\leq j_1<\cdots< j_\ell\leq k.$
Now, by using the inductive hypothesis and by applying Depth Lemma in the first $k-2$ above short exact sequences from top to bottom, step by step, we obtain $\depth(U_1)\geq t-1, \depth(U_2)\geq t-2,\ldots,\depth(U_{k-2})\geq t-k+2.$ Finally, by applying Depth Lemma in the last short exact sequence, since the depth of the middle term is $\geq t-k+2,$ we get $\depth(S/(P_{F_1}+\cdots +P_{F_k}))=|F_1\cap\cdots\cap F_k|\geq t-k+1.$

(f)$\Rightarrow$(d): Let $\Gamma$ be a subcomplex of $\Delta$ with $\Fc(\Gamma)=\{F_{j_1},\ldots,F_{j_k}\}\subset \Fc(\Delta)$. We have to show that 
$\depth(K[\Gamma])\geq t.$ We may obviously assume that $k<r$ and the facets of $\Gamma$ are $F_1,\ldots,F_k.$ If $k\leq t,$ then we use the short exact sequences derived from (\ref{sequence}) in the proof of (d) $\Rightarrow$ (f) and, by applying successively  Depth Lemma from bottom to the top, we get, step by step, $\depth(U_{k-2})\geq t-k+2,\ldots, \depth(U_2)\geq t-2, \depth(U_1)\geq t-1$, and, finally, from the first exact sequence, $\depth(K[\Gamma])\geq t.$ If $t<k,$ we use only the first $t$ short exact sequences, that is, we stop at
\begin{equation}\nonumber
0\to U_{t-1}\to \Dirsum\limits_{1\leq j_1<\cdots < j_{t}\leq k}\frac{S}{P_{F_{j_1}}+\cdots + P_{F_{j_t}}}\to U_t\to 0.
\end{equation}
Since the middle term in this short exact sequence has $\depth\geq 1,$ we get $\depth(U_{t-1})\geq 1.$ Next, by using the same arguments as before, we get $\depth(U_{t-2})\geq 2,\ldots,\depth(U_{1})\geq t-1$, and, finally, $\depth(K[\Delta])\geq t,$ as desired.

The implication (d) $\Rightarrow$ (a) follows by Theorem~\ref{extension}.

Finally, the implication (c) $\Rightarrow$ (e) follows  similarly to the proof of Corollary 1.9 in \cite{MT}.
\end{proof}

In order to state the first consequence of the above theorem, we need to know the behavior of the depth of a Stanley-Reisner ring over a field when passing from characteristic $0$ to characteristic $p>0.$ We show in the next lemma that the Betti numbers of the Stanley-Reisner ring  can only go up when passing from characteristic $0$ to a positive characteristic which, in particular, implies that the depth does not increase. This result is certainly known. However we include here its proof since we could not find any precise reference. The argument of the proof was communicated to the second author by Ezra Miller.

\begin{Lemma}\label{betti}
Let $\Delta$ be a simplicial complex on the vertex set $[n]$ and let $K,L$ be two fields with $\car K=0$, $\car L=p>0.$ Then
$
\beta_i(K[\Delta])\leq \beta_i(L[\Delta]) \text{ for all }i.
$
\end{Lemma}

\begin{proof}
Any field is flat over its prime field. Therefore,  since $\car K=0,$ we have $\beta_i(K[\Delta])=\beta_i(\QQ[\Delta])$ for all $i,$ and since $\car L=p$, we have 
$\beta_i(K[\Delta])=\beta_i(\FF_p[\Delta])$ for all $i,$ where $\FF_p$ is the prime field of characteristic $p.$ In other words, the Betti numbers depend only on 
the characteristic of the base field. Let $\ZZ_p$ be  the local ring of the integers at the
prime $p$. The ring $\ZZ_p[X]$ is *local (\cite[Section 1.5]{BH}) and the Stanley-Reisner ideal $I_\Delta\subset \ZZ_p[X]$ is *homogeneous. Let 
$\Fc$ be a minimal free resolution of $\ZZ_p[\Delta]$ over $\ZZ_p[x_1,\ldots,x_n].$ Since $p$ is a nonzerodivisor on $\ZZ_p[\Delta]$, by \cite[Lemma 8.27]{MS}, the quotient 
$\Fc/p\Fc$ is a minimal free resolution of $\FF_p[\Delta]$ over $\FF_p[x_1,\ldots,x_n]$. On the other hand, the localization $\Fc[p^{-1}]$ by inverting $p$ is a free resolution, not necessarily minimal, of $\QQ[\Delta]$ over $\QQ[x_1,\ldots,x_n]$. Since the modules in $\Fc/p\Fc$ and $\Fc[p^{-1}]$ have the same ranks, it follows that $\beta_i(\QQ[\Delta])\leq \beta_i(\FF_p[\Delta])$ for all $i$ which leads to the desired inequalities. 
\end{proof}

\begin{Corollary}\label{indep}
Let $\Delta$ be a pure simplicial complex with rigid depth over a field  of characteristic $0$. Then $\Delta$ has rigid depth over any field.
\end{Corollary}

\begin{proof} Let $K$ be a field of characteristic $0$ and $L$ a field of characteristic $p>0.$ The above lemma implies that $\projdim K[\Delta]\leq \projdim L[\Delta].$ By Auslander-Buchsbaum formula, it follows that $\depth K[\Delta]\geq \depth L[\Delta].$ Therefore, the desired statement follows by applying the combinatorial condition (f) of Theorem~\ref{rigiddepth}.
\end{proof}

\begin{Example}{\em 
Let $\Delta$ be the six-vertex triangulation of the real projective plane; see \cite[Section 5.3]{BH}. If $\chara K\neq 2,$ 
then $\Delta$ is Cohen-Macaulay over $K$, hence $\depth(K[\Delta])=2$, and, by condition (f) of Theorem~\ref{rigiddepth}, it follows that $\Delta$ 
does not have rigid depth over $K.$ But if $\chara K=2,$ then $\depth(K[\Delta])=1,$  and, consequently, $\Delta$ has rigid depth over $K.$
}
\end{Example}

The simplicial complexes with one or two facets have rigid depth. 

\begin{Lemma}\label{2facets}
Let $\Delta$ be a pure simplicial complex with at most two facets. Then $\Delta$ has rigid depth.
\end{Lemma}

\begin{proof} We only need to consider the case of simplicial complexes with two facets since the other case is obvious. Let $\dim \Delta=d-1$ and $\Fc(\Delta)=\{F,G\}$.
We show that $\depth(K[\Delta])=t$ if and only if $|F\cap G|=t-1.$ Then the claim follows by condition (f) in Theorem~\ref{rigiddepth}.
We consider the exact sequence
\[
0\to K[\Delta]\to (S/P_F)\oplus(S/P_G)\to S/(P_F+P_G)\cong K[x_i\ |\ i\in F\cap G]\to 0.
\]
As $(S/P_F)\oplus(S/P_G)$ and $S/(P_F+P_G)$ are Cohen-Macaulay of dimensions $d$ and, respectively, $|F\cap G|$, it follows that $\depth(K[\Delta])=t$
if and only if $|F\cap G|=t-1.$
\end{proof}

\begin{Example}{\em 
Let $\Delta$ and $\Gamma$ be the simplicial complexes with $\Fc(\Delta)=\{\{1,2,3\},$ $\{1,4,5\}\}$ and $\Fc(\Gamma)=\{\{1,2,3\},\{1,3,4\}\}$. Obviously, 
by Lemma~\ref{2facets}, $\Delta$ is non-Cohen-Macaulay of rigid depth $2$, while $\Gamma$ is Cohen-Macaulay of rigid depth. 
}\end{Example}

In the sequel we investigate whether the rigid depth property is preserved by the skeletons of the simplicial complexes with rigid depth. The next example  shows that this is not the case.

\begin{Example}{\em
Let $\Delta$ be the simplicial complex on the vertex set $[8]$ with $\Fc(\Delta)=\{F,G\}$ where $F=\{1,2,3,4,5\}$ and $G=\{1,2,6,7,8\}.$ Then,
by Lemma~\ref{2facets} and its proof, it follows that $\depth(K[\Delta])=3$ and $\Delta$ has rigid depth. Let $\Delta^{(3)}$ be the $3$-dimensional skeleton of 
$\Delta$ and $\Gamma$ the subcomplex of $\Delta^{(3)}$ with the facets $G_1=\{1,2,3,5\}$ and $G_2=\{2,6,7,8\}.$ Then, again by the proof the above lemma, we get $\depth(K[\Gamma])=2$. But $\depth K[\Delta^{(3)}]=3$, thus the skeleton $\Delta^{(3)}$ of $\Delta$ does not have rigid depth since it does not satisfy condition (d) in Theorem~\ref{rigiddepth}.
}
\end{Example}

However, as an application of Theorem~\ref{rigiddepth}, we prove the following

\begin{Proposition}
Let $\Delta$ be a pure simplicial complex with rigid depth and let $t=\depth(K[\Delta])$. If $\Delta^{(i)}$ has rigid depth for some $i\geq t-1$,
then $\Delta^{(j)}$ has rigid depth for every $j\geq i.$
\end{Proposition}

\begin{proof}
By \cite{JAX}, we know that $\depth(K[\Delta^{(i)}])=t$ for $i\geq t-1.$ It is enough to show that if $\Delta^{(i)}$ has rigid depth  for some $i\geq t-1,$ then $\Delta^{(i+1)}$ has the same property. 

Let $\Gamma\subset \Delta^{(i+1)}$ be a subcomplex with $\Fc(\Gamma)\subset \Fc(\Delta^{(i+1)}).$   Then $\Gamma^{(i)}$ is a subcomplex of $\Delta^{(i)}$ and $\Fc(\Gamma^{(i)})\subset \Fc(\Delta^{(i)}).$ By our assumption and by using condition (e) in Theorem~\ref{rigiddepth}, it follows that 
$\Gamma^{(t-1)}$ is Cohen-Macaulay. Therefore, $\Delta^{(i+1)}$ satisfies  condition (e) in Theorem~\ref{rigiddepth}, which ends our proof.
\end{proof}

\begin{center}
{\bf Acknowledgment}
\end{center}

We thank J\"urgen Herzog for helpful discussions on the subject of this paper and Ezra Miller for  the proof of Lemma~\ref{betti}. 
We would also like to thank the referee for his valuable suggestions to improve our paper.

\end{document}